\documentclass[a4paper,12pt,leqno]{amsart}
\usepackage{amsmath,amssymb}
\usepackage{amsthm}
\usepackage{bm}

\numberwithin{equation}{section}

\newtheorem{theorem}{Theorem}[section]

\newtheorem{corollary}[theorem]{Corollary}
\theoremstyle{definition}

\newtheorem{example}[theorem]{Example}
\newtheorem{remark}[theorem]{Remark}

\allowdisplaybreaks

\DeclareMathOperator*{\Res}{\mathrm{Res}}

\def\abs#1{\lvert#1\rvert}
\begin{document}

\title[Barnes multiple zeta-functions]
{Barnes multiple zeta-functions, Ramanujan's formula, and relevant series involving hyperbolic functions}

\author{Yasushi Komori,\ Kohji Matsumoto and Hirofumi Tsumura}

\maketitle

\begin{abstract}
  In the former part of this paper, we give functional equations for
  Barnes multiple zeta-functions and consider some relevant results. 
  In particular, we show that Ramanujan's classical formula for the
  Riemann zeta values can be derived from functional equations for
  Barnes zeta-functions.   In the
  latter half part, we generalize some evaluation formulas of certain
  series involving hyperbolic functions in terms of Bernoulli
  polynomials. The original formulas were classically given by Cauchy,
  Mellin, Ramanujan, and later recovered and formulated by Berndt.
  From our consideration, we give multiple versions of these known
  formulas.

\noindent
{\sc{MSC Numbers}}: 11M41,\,11B68.

\end{abstract}

%%%%%%%%%%%%%%%%%%%%%%%%%%%%%%%%%%%%%%%%%%%%%%%
\section{Introduction} \label{sec-1}
%%%%%%%%%%%%%%%%%%%%%%%%%%%%%%%%%%%%%%%%%%%%%%%

Let $\mathbb{N}$ be the set of natural numbers, $\mathbb{Z}$ the ring
of rational integers, $\mathbb{Q}$ the field of rational numbers,
$\mathbb{R}$ the field of real numbers, $\mathbb{C}$ the field of
complex numbers, and $\mathbb{N}_0:=\mathbb{N}\cup\{0\}$.

We begin with the classical work of Cauchy \cite{Ca} who studied the
series defined by
\begin{equation}
  \sum_{m \in \mathbb{Z}\setminus \{0\}} \frac{(-1)^{m}}{\sinh(m\pi)m^{s}}\ \ \ (s \in \mathbb{Z}), \label{eq-1-1}
\end{equation}
where $\sinh x=(e^x-e^{-x})/2$. He showed that several values at
$s=4k+3$ $(k \in \mathbb{N}_0)$ can be written in terms of
$\pi$. After his work, this series was considered by Mellin,
Ramanujan, and several other authors (see \cite{Be2,Be3,Me1,Me2}), and
the following fascinating formula was proved:
\begin{equation}
  \sum_{m \in \mathbb{Z}\setminus \{0\}} \frac{(-1)^{m}}{\sinh(m\pi)m^{4k+3}} ={(2\pi)^{4k+3}}\sum_{j=0}^{2k+2}(-1)^{j+1} \frac{B_{2j}(1/2)}{(2j)!}\frac{B_{4k+4-2j}(1/2)}{(4k+4-2j)!}  \label{eq-1-2}
\end{equation}
for $k \in \mathbb{N}_0$, where $B_j(y)$ is the $j$th Bernoulli
polynomial defined by
\begin{equation}
  F(t,y)= \frac{te^{ty}}{e^t-1}=\sum_{j=0}^\infty B_j(y)\frac{t^j}{j!} \label{Ber-poly}
\end{equation}
(see \cite{Di}).  As a result related to \eqref{eq-1-2}, it is also
known that
\begin{equation}
  \sum_{m\in \mathbb{Z}\setminus \{0\}}\frac{\coth(m\pi)}{m^{4k+3}}={(2\pi)^{4k+3}}\sum_{j=0}^{2k+2}(-1)^{j+1} \frac{B_{2j}(0)}{(2j)!}\frac{B_{4k+4-2j}(0)}{(4k+4-2j)!}  \label{eq-1-3}
\end{equation}
for $k \in \mathbb{N}_0$, which is written in Ramanujan's notebooks
(see Berndt \cite[(25.3) p.\,293]{Be2}), where $\coth
x=(e^x+e^{-x})/(e^x-e^{-x})$. In fact, \eqref{eq-1-3} can be easily
derived from Ramanujan's famous formula (see Berndt
\cite[p.\,275]{Be2}):
\begin{equation}
  \begin{split}
    & \alpha^{-N}\left\{ \frac{1}{2}\zeta(2N+1)+\sum_{k=1}^\infty \frac{1}{\left(e^{2k\alpha}-1\right)k^{2N+1}}\right\} \\
    & =(-\beta)^{-N}\left\{ \frac{1}{2}\zeta(2N+1)+\sum_{k=1}^\infty \frac{1}{\left(e^{2k\beta}-1\right)k^{2N+1}}\right\} \\
    & \ \
    -2^{2N}\sum_{k=0}^{N+1}(-1)^{k}\frac{B_{2k}(0)}{(2k)!}\frac{B_{2N+2-2k}(0)}{(2N+2-2k)!}\alpha^{N+1-k}\beta^{k},
  \end{split}
  \label{eq-1-4}
\end{equation}
where $N$ is any non-zero integer, $\alpha$ and $\beta$ are positive
numbers such that $\alpha\beta=\pi^2$ and $\zeta(s)$ is the Riemann
zeta-function.

In the 1970's, Berndt \cite{BeTAMS,Be0} studied generalized Eisenstein series
and proved transformation formulas for them. Using this result, he
gave a family of evaluation formulas for certain Dirichlet series in
\cite{BeRMJM,Be1}, including \eqref{eq-1-2}, \eqref{eq-1-3} and \eqref{eq-1-4} 
(see also Remark \ref{R-1}).

What is the meaning of the above infinite series involving hyperbolic
functions?    We can find that they are connected with Barnes multiple
zeta-functions.   In fact,
in the former half part of this paper, we study functional equations
for the Barnes zeta-functions and we show two expressions of the
Barnes zeta-functions or their residues at integers.  We observe that
Ramanujan's formula \eqref{eq-1-4} (and hence \eqref{eq-1-3}) can be
deduced by combining these two expressions in the double case.  Hence
in the multiple cases, the combination of these expressions may be
regarded as generalizations of Ramanujan's formula (see Corollary
\ref{cor:spvals_Barnes_0}).  

Motivated by this observation, in the
latter half part, we first give a very general form of evaluation
formulas (see Theorem \ref{T-2-1}), which is out of the frame of Barnes
zeta-functions.   From this form, we deduce a
certain explicit evaluation formula with a parameter $y\in [0,1]$ (see
Theorem \ref{P-3-1}) which may be regarded as a relation of several
Barnes zeta-functions at non-positive integers.  This formula
especially implies \eqref{eq-1-2} and \eqref{eq-1-3} (see Corollaries
\ref{C-3-1} and \ref{C-3-2}) and also implies a lot of presumably new
formulas, for example,
\begin{align}
 & \sum_{m\in \mathbb{Z}\setminus \{0\}}\frac{1}{\sinh(m\pi i/\rho)^2
    m^4} =-\frac{1}{2835}\pi^4, \label{eq-1-6}\\
  & \sum_{m\in \mathbb{Z}\setminus \{0\}}\frac{\coth(m\pi i/\rho)^2 }{m^4} =\frac{62}{2835}\pi^4, \label{eq-1-5}
\end{align}
where $i=\sqrt{-1}$ and $\rho=(-1+\sqrt{-3})/2$, the cube root of unity,
and the same type of formulas including higher power roots of unity
(see Corollary \ref{C-3-3}, Example \ref{Ex-3-4}).

%%%%%%%%%%%%%%%%%%%%%%%%%%%%%%%%%%%%%%%%%%%%%%%%%%%%%%%%%%%%%%%%%%%%%%%%%%%

%%%%%%%%%%%%%%%%%%%%%%%%%%%%%%%%%%%%%%%%%%%%%%%
\section{Functional equations for Barnes zeta-functions} \label{sec-1a}
%%%%%%%%%%%%%%%%%%%%%%%%%%%%%%%%%%%%%%%%%%%%%%%
For $\theta\in\mathbb{R}$ let
$H(\theta)=\{z=re^{i(\theta+\phi)}\in\mathbb{C}~|~r>0,-\pi/2<\phi<\pi/2\}$
be the open half plane whose normal vector is $e^{i\theta}$.  We
recall the Barnes zeta-function defined by the following multiple
Dirichlet series:
\begin{equation}
  \label{eq:def_Barnes}
  \zeta_n(s,a;\omega_1,\ldots,\omega_n)
  =\sum_{m_1=0}^\infty\cdots\sum_{m_n=0}^\infty \frac{1}{(a+\omega_1m_1+\cdots+\omega_nm_n)^s},
\end{equation}
where all $a,\omega_1,\ldots,\omega_n\in H(\theta)$ for some $\theta$.
Then it is known that this Dirichlet series converges absolutely
uniformly on any compact subset in $\Re s>n$.

Assume at first that $\Re s>n$.  For $x\in H(\theta)$, we have the
formula for the gamma function
\begin{equation}
  \label{eq:formula_gamma}
  x^{-s}=\frac{1}{\Gamma(s)}\int_0^{e^{-i\theta}\infty}e^{-xt}t^{s-1}dt.
\end{equation}
Since 
\begin{equation}
  a+\omega_1m_1+\cdots+\omega_nm_n\in H(\theta)
\end{equation}
for $m_1,\ldots,m_n\in\mathbb{N}$, we can apply
\eqref{eq:formula_gamma} to each term in \eqref{eq:def_Barnes} to get
\begin{equation}
  \begin{split}
    \zeta_n(s,a;\omega_1,\ldots,\omega_n)
    &=\sum_{m_1=0}^\infty\cdots\sum_{m_n=0}^\infty
    \frac{1}{\Gamma(s)}\int_0^{e^{-i\theta}\infty}e^{-(a+\omega_1m_1+\cdots+\omega_nm_n)t}t^{s-1}dt
    \\
    &=\frac{1}{\Gamma(s)}
    \int_0^{e^{-i\theta}\infty}
    \frac{e^{(\omega_1+\cdots+\omega_n-a)t}}
    {(e^{\omega_1t}-1)\cdots(e^{\omega_nt}-1)}
    t^{s-1}dt
    \\
    &=\frac{1}{\Gamma(s)(e^{2\pi is}-1)}
    \int_{C(\theta)}
    \frac{e^{(\omega_1+\cdots+\omega_n-a)t}}
    {(e^{\omega_1t}-1)\cdots(e^{\omega_nt}-1)}
    t^{s-1}dt,
  \end{split}
\end{equation}
where the argument of $t$ is taken in $-\theta\leq\arg
t\leq-\theta+2\pi$ and $C(\theta)$ is a contour which starts at
$e^{-i\theta}\infty$, goes counterclockwise around the origin with
sufficiently small radius, and ends at $e^{-i\theta}\infty$.  
Let
$0\leq y_1,\ldots,y_n<1$ and put 
$$
a=a(y_1,\ldots,y_n)=\omega_1(1-y_1)+\cdots+\omega_n(1-y_n)\in H(\theta).
$$   
Then
\begin{equation}
\label{eq:int_rep_Barnes}
  \begin{split}
    &\zeta_n(s,a(y_1,\ldots,y_n);\omega_1,\ldots,\omega_n)\\
    &\qquad=\frac{1}{\Gamma(s)(e^{2\pi is}-1)}
    \int_{C(\theta)}
    \frac{e^{(\omega_1 y_1+\cdots+\omega_n y_n)t}}
    {(e^{\omega_1t}-1)\cdots(e^{\omega_nt}-1)}
    t^{s-1}dt
    \\
    &\qquad=\frac{\prod_{j=1}^n \omega_j^{-1}}{\Gamma(s)(e^{2\pi is}-1)}
    \int_{C(\theta)}
    \Bigl(
    \prod_{j=1}^n F(\omega_j t,y_j)
    \Bigr)
    t^{s-n-1}dt.
  \end{split}
\end{equation}
% Let
% $0\leq y<1$ and put
% \begin{equation*}
% a=a(y)=(\omega_1+\cdots+\omega_n)(1-y)\in H(\theta).
% \end{equation*}
% % Let
% % $0\leq y_1,\ldots,y_n<1$ and put
% % \begin{equation*}
% % a=a(y_1,\ldots,y_n)=\omega_1(1-y_1)+\cdots+\omega_n(1-y_n)\in H(\theta).
% % \end{equation*}
% Then
% \begin{equation}
% \label{eq:int_rep_Barnes}
% \begin{split}
%     \zeta_n(s,a(y);\omega_1,\ldots,\omega_n)%\\
% %    &\zeta_n(s,a(y_1,\ldots,y_n);\omega_1,\ldots,\omega_n)\\
%     &=\frac{1}{\Gamma(s)(e^{2\pi is}-1)}
%     \int_{C(\theta)}
%     \frac{e^{(\omega_1+\cdots+\omega_n) yt}}
%     {(e^{\omega_1t}-1)\cdots(e^{\omega_nt}-1)}
%     t^{s-1}dt
%     \\
%     &=\frac{\prod_{j=1}^n \omega_j^{-1}}{\Gamma(s)(e^{2\pi is}-1)}
%     \int_{C(\theta)}
%     \Bigl(
%     \prod_{j=1}^n F(\omega_j t,y)
%     \Bigr)
%     t^{s-n-1}dt.
%   \end{split}
% \end{equation}
If $t\in C(\theta)$ is sufficiently far from the origin, then
$\Re(\omega_j t)>0$ ($1\leq j\leq n$).  Therefore the integral on the
rightmost side converges absolutely uniformly on the whole space
$\mathbb{C}$, so \eqref{eq:int_rep_Barnes} gives the meromorphic
continuation of
$\zeta_n(s,a(y);\omega_1,\ldots,\omega_n)$ to the whole
space $\mathbb{C}$.

In the following, we assume that $n\geq2$ and $\Im
(\omega_j/\omega_k)\neq0$ for any pair $(j,k)$ with $j\neq k$.  
From
the above integral expression we obtain the following functional
equations for Barnes zeta-functions.  
When $y_1=\cdots=y_n=y$, we write $a(y_1,\ldots,y_n)=a(y)$ for brevity.
\begin{theorem}[functional equations]
\label{thm:funceq_Barnes}
We have
\begin{multline}
  \label{eq:funceq_Barnes}
  \zeta_n(s,a(y);\omega_1,\ldots,\omega_n) 
  \\
  =-
  \frac{2\pi i}{\Gamma(s)(e^{2\pi is}-1)} 
  \sum_{k=1}^n
  \sum_{m\in\mathbb{Z}\setminus\{0\}} 
  \omega_k^{-1}
  \Bigl( 
  \prod_{\substack{j=1\\j\neq k}}^n 
  \frac{e^{(2m\pi i\omega_j/\omega_k)y}}
  {e^{2m\pi i\omega_j/\omega_k}-1}
  \Bigr) (2m\pi i\omega_k^{-1})^{s-1}e^{2m\pi i y},
\end{multline}
where the right-hand side converges absolutely uniformly on the whole space $\mathbb{C}$
if $0<y<1$, and on the region $\Re s<0$ if $y=0$.
%and hence this expression gives the meromorphic continuation of the Barnes 
%eta-function with respect to $s$.

In particular, if $y=1/2$, we have
\begin{multline}
  \label{eq:funceq_Barnes_omake}
  \zeta_n(s,(\omega_1+\cdots+\omega_n)/2;\omega_1,\ldots,\omega_n) 
  \\
  =-
  \frac{1}{2^{n-1}}
  \frac{2\pi i}{\Gamma(s)(e^{2\pi is}-1)} 
  \sum_{k=1}^n
  \sum_{m\in\mathbb{Z}\setminus\{0\}} 
  \omega_k^{-1}
  (-1)^m
  \Bigl( 
  \prod_{\substack{j=1\\j\neq k}}^n 
  \frac{1}
  {\sinh(m\pi i\omega_j/\omega_k)}
  \Bigr) (2m\pi i\omega_k^{-1})^{s-1}.
\end{multline}
\end{theorem}
In the case $y=0$ of Theorem \ref{thm:funceq_Barnes}, the series expression
\eqref{eq:funceq_Barnes} is valid only for $\Re s<0$.  In
order to remove this restriction, we decompose the series into the
terms involving the
Riemann zeta-function and the remaining parts.  
For $k\in\{1,\ldots,n\}$, let
$I_k^+=\{j\in\{1,\ldots,n\}\setminus\{k\}|\Im(\omega_j/\omega_k)>0\}$
and
$I_k^-=\{j\in\{1,\ldots,n\}\setminus\{k\}|\Im(\omega_j/\omega_k)<0\}$.
Let
\begin{equation}
  \delta(J)=
  \begin{cases}
    0\qquad&(J\neq\emptyset)\\
    (-1)^{n+1}\qquad&(J=\emptyset)
  \end{cases}
\end{equation}
for $J\subset\{1,\ldots,n\}$.

% Let $k\in\{1,\ldots,n\}$.
% For any set $J_k\subset\{1,\ldots,n\}\setminus\{k\}$, we denote
% $J_k^c=\{1,\ldots,n\}\setminus (J_k\cup\{k\})$ and put
% \begin{equation}
%   \delta(J_k)=
%   \begin{cases}
%     0\qquad&(J_k\neq\emptyset)\\
%     (-1)^{n+1}\qquad&(J_k=\emptyset).
%   \end{cases}
% \end{equation}
% Let
% $I_k=\{j\in\{1,\ldots,n\}\setminus\{k\}|\Im(\omega_j/\omega_k)>0\}$.

\begin{corollary}
\label{thm:funceq_Barnes_0}
We have
\begin{multline}
  \zeta_n(s,a(0);\omega_1,\ldots,\omega_n) 
  \\
  \begin{aligned}
  &=-\frac{2\pi i}{\Gamma(s)(e^{2\pi is}-1)}
    \sum_{k=1}^n
    \omega_k^{-1}
    \Biggl\{
    \sum_{m>0}
    (2m\pi i\omega_k^{-1})^{s-1}
    \biggl(
    \Bigl(
    \prod_{\substack{j=1\\j\neq k}}^n
    \frac{1}{e^{2m\pi i\omega_j/\omega_k}-1}
    \Bigr)
    -\delta(I_k^-)
    \biggr)
%    \Biggr.
    \\
    &\qquad
    +
    \sum_{m>0}
    (-2m\pi i\omega_k^{-1})^{s-1}
    \biggl(
    \Bigl(
    \prod_{\substack{j=1\\j\neq k}}^n
    \frac{1}{e^{-2m\pi i\omega_j/\omega_k}-1}
    \Bigr)
    -\delta(I_k^+)
    \biggr)
    \\
    &\qquad
%    \Biggl.
    +
    \delta(I_k^-)
    (2\pi i\omega_k^{-1})^{s-1}
    \zeta(1-s)
%     \\
%     &\qquad
    +
    \delta(I_k^+)
    (-2\pi i\omega_k^{-1})^{s-1}
    \zeta(1-s)
    \Biggr\},
  \end{aligned}
\end{multline}
where the series in the right-hand side converge absolutely uniformly on the whole space $\mathbb{C}$.
\end{corollary}

Proofs of Theorem \ref{thm:funceq_Barnes} and Corollary \ref{thm:funceq_Barnes_0}
will be given in Section 4.

In the following, the empty sum should be understood as $0$.
\begin{corollary}
\label{cor:spvals_Barnes}
For $l\in\mathbb{Z}$ {\rm(}or $l>n$ if $y=0${\rm)}, we have
\begin{multline}
  \label{eq:spvals_Barnes}
  \sum_{k=1}^n
  \sum_{m\in\mathbb{Z}\setminus\{0\}} 
  \omega_k^{-1}
  \Bigl( 
  \prod_{\substack{j=1\\j\neq k}}^n 
  \frac{e^{(2m\pi i\omega_j/\omega_k)y}}
  {e^{2m\pi i\omega_j/\omega_k}-1}
  \Bigr) (2m\pi i\omega_k^{-1})^{n-l-1}e^{2m\pi i y}
\\
=
\begin{cases}
  \dfrac{(-1)^{l-n+1}}{(l-n)!}
  \zeta_n(n-l,a(y);\omega_1,\ldots,\omega_n) \qquad&(l\geq n)\\%[5truemm]
  -(n-l-1)!\Res_{s=n-l}
  \zeta_n(s,a(y);\omega_1,\ldots,\omega_n) \qquad&(l< n).
\end{cases}
\end{multline}
On the other hand, this is equal to
\begin{equation}
\label{eq:Ber}
  -
\sum_{\substack{m_1,\ldots,m_n\geq0\\m_1+\cdots+m_n=l}}
\prod_{j=1}^n\frac{B_{m_j}(y)}{m_j!}\omega_j^{m_j-1}.
\end{equation}
\end{corollary}
\begin{proof}
  For $k\in\mathbb{Z}$, the expansion
  \begin{equation}
    \label{gamma-residue}
    \frac{\Gamma(s)(e^{2\pi is}-1)}{2\pi i}=
    \begin{cases}
\dfrac{(-1)^k }{(-k)!}+O(s-k)\qquad&(k\leq 0)\\
(k-1)!(s-k)+O((s-k)^2)\qquad&(k>0)
\end{cases}
\end{equation}
%   \begin{equation}
%     \label{gamma-residue}
%     \lim_{s\to-k}\frac{\Gamma(s)(e^{2\pi is}-1)}{2\pi i}=\frac{(-1)^k }{k!},
%   \end{equation}
  holds.   Using this 
  and Theorem \ref{thm:funceq_Barnes}, we obtain
  \eqref{eq:spvals_Barnes}.  On the other hand, by use of the integral representation
  \eqref{eq:int_rep_Barnes}, we see that the left-hand side of
  \eqref{eq:spvals_Barnes} is equal to
\begin{equation}
\label{eq:res_Barnes}
%   \frac{(-1)^{l-n+1}}{(l-n)!}
%     \zeta_n(n-l,a(y);\omega_1,\ldots,\omega_n) 
%     \\
%     =
    -
    \Bigl(\prod_{j=1}^n\omega_j^{-1}\Bigr)
    \Res_{t=0}
    \left\{\Bigl(
    \prod_{j=1}^n F(\omega_j t,y)
    \Bigr)
    t^{-l-1}\right\},
\end{equation}
which yields \eqref{eq:Ber}.
\end{proof}
Similarly we have the following.
\begin{corollary}
\label{cor:spvals_Barnes_0}
For $l\in\mathbb{Z}\setminus\{n\}$, we have
\begin{equation}
\label{eq:spvals_Barnes_0}
  \begin{split}
    &\sum_{k=1}^n
    \omega_k^{-1}
%   \\
%   &\times
    \Biggl(
    \sum_{m>0}
    (2m\pi i\omega_k^{-1})^{n-l-1}
    \biggl(
    \Bigl(
    \prod_{\substack{j=1\\j\neq k}}^n
    \frac{1}{e^{2m\pi i\omega_j/\omega_k}-1}
    \Bigr)
    -\delta(I_k^-)
    \biggr)
    \\
    &\qquad+
    \sum_{m>0}
    (-2m\pi i\omega_k^{-1})^{n-l-1}
    \biggl(
    \Bigl(
    \prod_{\substack{j=1\\j\neq k}}^n
    \frac{1}{e^{-2m\pi i\omega_j/\omega_k}-1}
    \Bigr)
    -\delta(I_k^+)
    \biggr)
%\Biggr)
    \\
    &\qquad
    + 
%{\delta_{l\neq n}}
%     \sum_{k=1}^n
%     \omega_k^{-1}
%    \Biggl(
    \delta(I_k^-)
    (2\pi i\omega_k^{-1})^{n-l-1}
    \zeta(1-n+l)
    \\
    &\qquad
    +
    \delta(I_k^+)
    (-2\pi i\omega_k^{-1})^{n-l-1}
    \zeta(1-n+l)
    \Biggr)
%    \\
  \\
  &=
  \begin{cases}
    \dfrac{(-1)^{l-n+1}}{(l-n)!}
    \zeta_n(n-l,a(0);\omega_1,\ldots,\omega_n) \qquad&(l>n)\\%[5truemm]
    -(n-l-1)!\Res_{s=n-l}
    \zeta_n(s,a(0);\omega_1,\ldots,\omega_n) \qquad&(l<n)
  \end{cases}
  \\
  &=
  -%\Bigl(\prod_{j=1}^n\omega_j^{-1}\Bigr)
  \sum_{\substack{m_1,\ldots,m_n\geq0\\m_1+\cdots+m_n=l}}
  \prod_{j=1}^n\frac{B_{m_j}(0)}{m_j!}\omega_j^{m_j-1}.
\end{split}
\end{equation}
\end{corollary}

In the next section we will show that Ramanujan's formula \eqref{eq-1-4} 
is a consequence of the case $n=2$ of \eqref{eq:spvals_Barnes_0}.
Therefore Corollary \ref{cor:spvals_Barnes_0} can be regarded as 
generalizations of Ramanujan's formula.

Here we give historical remarks.  A kind of functional equations for
the Barnes zeta-functions was first proved by Hardy and Littlewood
\cite{Ha} in the case $n=2$, and a generalization to the case of
general $n$ was discussed in Egami's lecture note \cite{Eg}.  Our
proof of Theorem \ref{thm:funceq_Barnes} is essentially the same as
those of them.  On the other hand, by calculating explicitly the
residue %in the right-hand side of 
\eqref{eq:res_Barnes}, we showed an
expression of $\zeta_n(n-l,a(y);\omega_1,\ldots,\omega_n)$ or its residues in terms of
Bernoulli polynomials.  This type of results is also classical,
already studied by Barnes himself \cite{B01}, \cite{B04} (see also
\cite{Ar}, \cite{Ma98}).  In this sense, both of the two equalities in
\eqref{eq:spvals_Barnes} and \eqref{eq:Ber} are classical.  The novel
point in the present paper is to combine these two equalities.  
A consequence of such a combination is the observation concerning Ramanujan's
formula in the next section.

%%%%%%%%%%%%%%%%%%%%%%%%%%%%%%%%%%%%%%%%%%%%%%%
\section{Ramanujan's formula} \label{sec-1b}
%%%%%%%%%%%%%%%%%%%%%%%%%%%%%%%%%%%%%%%%%%%%%%%
In this section, we show that Ramanujan's formula
\eqref{eq-1-4} can be obtained by combining two equalities given in
Corollary \ref{cor:spvals_Barnes_0}.
In Corollary \ref{cor:spvals_Barnes_0}, consider the
case $n=2$, $\omega_1=\alpha^{1/2}$, $\omega_2=i\beta^{1/2}$ with
$\alpha,\beta\in\mathbb{R}$. Let  $N\in\mathbb{Z}\setminus\{0\}$.

The last member of \eqref{eq:spvals_Barnes_0} is
\begin{equation}
  C_l=
  -(i\alpha^{1/2}\beta^{1/2})^{-1}\sum_{k=0}^l i^k
  \frac{B_{l-k}(0)}{(l-k)!}
  \frac{B_k(0)}{k!}
  \alpha^{(l-k)/2} \beta^{k/2}.
\end{equation}
In particular, for $l=2N+2$,
we have
\begin{equation}
\label{eq:eq_l}
  \begin{split}
    C_{2N+2}
    &=-(i\alpha^{1/2}\beta^{1/2})^{-1}
    \sum_{j=0}^{2N+2} i^j
    \frac{B_{2N+2-j}(0)}{(2N+2-j)!}
    \frac{B_j(0)}{j!}
    \alpha^{(2N+2-j)/2} \beta^{j/2}
    \\
    &=-(i\alpha^{1/2}\beta^{1/2})^{-1}
    \sum_{k=0}^{N+1} (-1)^k
    \frac{B_{2N+2-2k}(0)}{(2N+2-2k)!}
    \frac{B_{2k}(0)}{(2k)!}
    \alpha^{N+1-k} \beta^{k},
  \end{split}
\end{equation}
where we have used $B_{2j+1}(0)=0$ (for $j\geq 1$).

Next we compute the first member of \eqref{eq:spvals_Barnes_0}.  Since
$\delta(I_1^+)=\delta(I_2^-)=0$ and $\delta(I_1^-)=\delta(I_2^+)=-1$, this
is equal to
\begin{equation}
  \begin{split}
    &
    \alpha^{-1/2}\sum_{m>0}(2m\pi i\alpha^{-1/2})^{1-l}
    \biggl(\frac{1}{e^{-2m\pi \beta^{1/2}\alpha^{-1/2}}-1}+1\biggr)
    \\
    &\qquad
    +\alpha^{-1/2}\sum_{m>0}(-2m\pi i\alpha^{-1/2})^{1-l}
    \frac{1}{e^{2m\pi \beta^{1/2}\alpha^{-1/2}}-1}
    \\
    &
    -i\beta^{-1/2}\sum_{m>0}(2m\pi\beta^{-1/2})^{1-l}
    \frac{1}{e^{2m\pi\alpha^{1/2}\beta^{-1/2}}-1}
    \\
    &\qquad
    -i\beta^{-1/2}\sum_{m>0}(-2m\pi\beta^{-1/2})^{1-l}
    \biggl(\frac{1}{e^{-2m\pi \alpha^{1/2}\beta^{-1/2}}-1}+1\biggr)
    \\
    &-\alpha^{-1/2}(2\pi i\alpha^{-1/2})^{1-l}\zeta(l-1)
    +i\beta^{-1/2}(-2\pi\beta^{-1/2})^{1-l}\zeta(l-1).
  \end{split}
\end{equation}
In particular, in the case $\alpha^{1/2}\beta^{1/2}=\pi$,
we see that this is equal to
\begin{equation}
  \label{moshimoshikameyo}
  \begin{split}
    &
    -\alpha^{-1/2}\sum_{m>0}(2m\pi i\alpha^{-1/2})^{1-l}
    \frac{1}{e^{2m \beta}-1}
    \\
    &\qquad
    +\alpha^{-1/2}\sum_{m>0}(-2m\pi i\alpha^{-1/2})^{1-l}
    \frac{1}{e^{2m\beta}-1}
    \\
    &
    -i\beta^{-1/2}\sum_{m>0}(2m\pi\beta^{-1/2})^{1-l}
    \frac{1}{e^{2m\alpha}-1}
    \\
    &\qquad
    +i\beta^{-1/2}\sum_{m>0}(-2m\pi\beta^{-1/2})^{1-l}
    \frac{1}{e^{2m\alpha}-1}
    \\
    &-\alpha^{-1/2}(2\pi i\alpha^{-1/2})^{1-l}\zeta(l-1)
    +i\beta^{-1/2}(-2\pi\beta^{-1/2})^{1-l}\zeta(l-1)
    \\
    &=
    (2\pi)^{1-l}i^{1-l}\alpha^{(l-2)/2}
    \Bigl(-\zeta(l-1)+((-1)^{l-1}-1)
    \sum_{m=1}^\infty
    \frac{1}{(e^{2m\beta}-1)m^{l-1}}
    \Bigr)
    \\
    &\qquad
    -(2\pi)^{1-l}i\beta^{(l-2)/2}
    \Bigl((-1)^l\zeta(l-1)+(1-(-1)^{l-1})
    \sum_{m=1}^\infty
    \frac{1}{(e^{2m\alpha}-1)m^{l-1}}
    \Bigr).
  \end{split}
\end{equation}
Further in the case $l=2N+2$, we see that \eqref{moshimoshikameyo} reduces to
\begin{multline}
  \label{eq:eq_r}
  (\pi i)^{-1}(2\pi)^{-2N}(-\alpha)^N
  \Bigl(-\frac{1}{2}\zeta(2N+1)-
  \sum_{m=1}^\infty
  \frac{1}{(e^{2m\beta}-1)m^{2N+1}}
  \Bigr)
  \\
  +
  (\pi i)^{-1} (2\pi)^{-2N}\beta^N
  \Bigl(\frac{1}{2}\zeta(2N+1)+
  \sum_{m=1}^\infty
  \frac{1}{(e^{2m\alpha}-1)m^{2N+1}}
  \Bigr).
\end{multline}
By equating \eqref{eq:eq_l} and \eqref{eq:eq_r}, we finally obtain
\begin{multline}
%\label{eq:der_rama}
  -2^{2N}\sum_{k=0}^{N+1} (-1)^k
  \frac{B_{2N+2-2k}(0)}{(2N+2-2k)!}
  \frac{B_{2k}(0)}{(2k)!}
  \alpha^{N+1-k} \beta^k
  \\
  =
  (-\beta)^{-N}
  \Bigl(-\frac{1}{2}\zeta(2N+1)-
  \sum_{m=1}^\infty
  \frac{1}{(e^{2m\beta}-1)m^{2N+1}}
  \Bigr)
  \\
  +\alpha^{-N}
  \Bigl(\frac{1}{2}\zeta(2N+1)+
  \sum_{m=1}^\infty
  \frac{1}{(e^{2m\alpha}-1)m^{2N+1}}
  \Bigr),
\end{multline}
which recovers \eqref{eq-1-4}. % for $N>0$.
%coincides with a famous Ramanujan's formula.
% In the case $N<0$, the formula \eqref{eq-1-4} is obtained in a similar
% way.  We first multiply $\Gamma(s)(e^{2\pi is}-1)$ to
% \eqref{eq:int_rep_Barnes} and \eqref{eq:funceq_Barnes}. Then we
% rewrite the right-hand side of \eqref{eq:funceq_Barnes} into the same
% form as the rightmost of \eqref{moshimoshikameyo}, which yields its
% analytic continuation to the whole space.  Lastly putting
% $l=2N+2\in-2\mathbb{N}_0$, we obtain the formula.
% We see that the case $N<0$ corresponds to the residues of Barnes zeta-function.
% This construction can be generalized to the $n>2$ cases.

%From the above calculation, we see that
%Ramanujan's formula
%\eqref{eq-1-4}
%\eqref{eq:der_rama} is
%is essentially \eqref{eq:spvals_Barnes_0} in the case $n=2$.
%  we used some special properties for
% the case $n=2$.
%Hence
%Although this kind of calculation may not be applied to the general cases $n>2$, formula
%\eqref{eq:spvals_Barnes_0}
%can be regarded as generalizations of Ramanujan's formula.

%%%%%%%%%%%%%%%%%%%%%%%%%%%%%%%%%%%%%%%%%%%%%%%%%%%%%%%%%%%%%%%%%%%%%%%%%%%

\section{Proofs}\label{sec-p}

\begin{proof}[Proof of Theorem \ref{thm:funceq_Barnes}]
  For $z\in\mathbb{C}$ and
  $\varepsilon>0$, let $D(z,\varepsilon)$ be the closed disk whose
  center is $z$ with radius $\varepsilon$.  
    We use the notation $x_+=\max\{x,0\}$ for $x\in\mathbb{R}$.
  We first note that for any $\varepsilon>0$, there exists $M'=M'(\varepsilon)>0$ such that
  for $t\in\mathbb{C}\setminus\bigcup_{m\in\mathbb{Z}}D(2m\pi i,\varepsilon)$
  the inequality
  \begin{equation}
    \Bigl\lvert\frac{1}{e^t-1}\Bigr\rvert \leq M' e^{-(\Re t)_+}
  \end{equation}
  holds.
  Hence for
  $t\in\mathbb{C}\setminus\bigcup_{j=1}^n\bigcup_{m\in\mathbb{Z}}D(2m\pi i\omega_j^{-1},\varepsilon)$,
  we have
  \begin{equation}
    \Bigl\lvert
    \prod_{j=1}^n F(\omega_j t,y)
    \Bigr\rvert\leq 
    M \abs{t}^n e^{\sum_{j=1}^n(\Re (\omega_jt)y-\Re (\omega_jt)_+)}
  \end{equation}
with a certain $M=M(\varepsilon,\omega_1,\ldots,\omega_n)>0$.
  Since
  \begin{equation}
    \label{eq:ineq}
    \Re (\omega_jt)y-\Re (\omega_jt)_+\leq 0
  \end{equation}
  for $1\leq j\leq n$,
  we see that there exists $T=T(\varepsilon)\geq0$ such that for
  all $t\in\mathbb{C}$ with $|t|=1$,
  \begin{equation}
    \sum_{j=1}^n((\Re \omega_jt)y-(\Re \omega_jt)_+)\leq -T.
  \end{equation}
  Hence we see that for all 
  $t\in\mathbb{C}\setminus\bigcup_{j=1}^n\bigcup_{m\in\mathbb{Z}}D(2m\pi
  i\omega_j^{-1},\varepsilon)$,
  \begin{equation}\label{estim-M}
    \Bigl\lvert
    \prod_{j=1}^n F(\omega_j t,y)
    \Bigr\rvert\leq 
    M \abs{t}^n e^{-T|t|}.
  \end{equation}
  If $0<y<1$, then we can choose $T>0$.
In fact, since $\Im(\omega_j/\omega_k)\neq0$ for $j\neq k$ and $n\geq2$,
for any $t$ with $|t|=1$ we find at least one $j$ for which 
$\Re (\omega_jt)\neq 0$ holds.
Then, using $0<y<1$ we see that 
$\Re (\omega_jt)y-\Re (\omega_jt)_+ < 0$, from which $T>0$ easily follows. 

  From \eqref{estim-M}, we see that the integral on the rightmost side of
\eqref{eq:int_rep_Barnes}
  converges to $0$ when the radius of the contour goes to 
  infinity if $0<y<1$ or, $y=0$ with $\Re s<0$.
 Namely, there is a sequence $R_l\to\infty$ such that
  \begin{equation}
    \lim_{l\to\infty}\int_{|t|=R_l}    \Bigl\lvert
    \prod_{j=1}^n F(\omega_j t,y)
    \Bigr||t^{s-n-1}||dt|=0.
  \end{equation}

  Hence we can calculate the integral by counting
  all the residues on the whole
  space.  Since by the assumption the poles of the integrand is all
  simple except the origin, 
% and
%   \begin{equation}
%     \Res_{t=2m\pi i\omega_k^{-1}}F(\omega_k t,y)=(2m\pi i\omega_k^{-1})e^{2m\pi i y},
%   \end{equation}
  we obtain %\eqref{eq:funceq_Barnes}.
\begin{multline}
%  \label{eq:funceq_Barnes}
    \zeta_n(s,a(y);\omega_1,\ldots,\omega_n) 
    \\
    \begin{aligned}
      &=\frac{1}{\Gamma(s)(e^{2\pi is}-1)}
      \int_{C(\theta)}
      \frac{e^{(\omega_1+\cdots+\omega_n)yt}}
      {(e^{\omega_1t}-1)\cdots(e^{\omega_nt}-1)}
      t^{s-1}dt
      \\
      &=-
      \frac{2\pi i}{\Gamma(s)(e^{2\pi is}-1)} 
      \sum_{k=1}^n
      \sum_{m\in\mathbb{Z}\setminus\{0\}} 
      \omega_k^{-1}
      \Bigl( 
      \prod_{\substack{j=1\\j\neq k}}^n 
      \frac{e^{(2m\pi i\omega_j/\omega_k)y}}
      {e^{2m\pi i\omega_j/\omega_k}-1}
      \Bigr) (2m\pi i\omega_k^{-1})^{s-1}e^{2m\pi i y},
    \end{aligned}
  \end{multline}
whose
absolute and uniform convergence follows 
from the explicit form of
the series.
Therefore we obtain
 \eqref{eq:funceq_Barnes}.
\end{proof}

\begin{proof}[Proof of Corollary \ref{thm:funceq_Barnes_0}]
We observe that
as $m\to+\infty$,
$F(2\pi im\omega_j/\omega_k,0)=O(m)$ if $j\in I_k^+$
while 
$F(2\pi im\omega_j/\omega_k,0)$ decays exponentially if $j\in I_k^-$.
Thus we see that 
if $I_k^-\neq\emptyset$,
the series 
\begin{equation*}
\sum_{m>0} \Bigl( \prod_{\substack{j=1\\j\neq k}}^n \frac{1} {e^{2m\pi
    i\omega_j/\omega_k}-1} \Bigr) (2m\pi i\omega_k^{-1})^{s-1}
\end{equation*}
converges absolutely uniformly
for whole $s\in\mathbb{C}$.

Next consider the case $I_k^-=\emptyset$. 
We have
\begin{multline}
  \sum_{m>0}
  \Bigl( 
  \prod_{\substack{j=1\\j\neq k}}^n 
  \frac{1}
  {e^{2m\pi i\omega_j/\omega_k}-1}
  \Bigr) (2m\pi i\omega_k^{-1})^{s-1}
  \\
  \begin{aligned}
%     &=
%    \sum_{m>0} (-1)^{n-1}
%    \Bigl(
%    \prod_{\substack{j=1\\j\neq k}}^n
%    (2m\pi i\omega_j/\omega_k)
%    \bigl(1+
%    \frac{1}{e^{-2m\pi i\omega_j/\omega_k}-1}
%    \bigr)
%    \Bigr)
%    (2m\pi i\omega_k^{-1})^{s-1}
%    \\
   &=
    \sum_{m>0} (-1)^{n-1}
%    \Bigl(
%    \prod_{\substack{j=1\\j\neq k}}^n
%    \omega_j
%    \Bigr)
    (2m\pi i\omega_k^{-1})^{s-1}
    \Bigl(
    \prod_{\substack{j=1\\j\neq k}}^n
    \bigl(1+
    \frac{1}{e^{-2m\pi i\omega_j/\omega_k}-1}
    \bigr)
    \Bigr)
    \\
    &=
    \sum_{m>0} (-1)^{n-1}
%     \Bigl(
%     \prod_{\substack{j=1\\j\neq k}}^n
%     \omega_j
%     \Bigr)
    (2m\pi i\omega_k^{-1})^{s-1}
    \Biggl(1+
    \sum_{\substack{J\subset\{1,\ldots,n\}\setminus\{k\}\\|J|\geq1}}
    \Bigl(
    \prod_{j\in J}
    \frac{1}{e^{-2m\pi i\omega_j/\omega_k}-1}
    \Bigr)
    \Biggr)
    \\
    &=
    (-1)^{n-1}
%     \Bigl(
%     \prod_{\substack{j=1\\j\neq k}}^n
%     \omega_j
%     \Bigr)
    (2\pi i\omega_k^{-1})^{s-1}
    \zeta(1-s)
    \\
    &\qquad+
    (-1)^{n-1}
%     \Bigl(
%     \prod_{\substack{j=1\\j\neq k}}^n
%     \omega_j
%     \Bigr)
    \sum_{\substack{J\subset\{1,\ldots,n\}\setminus\{k\}\\|J|\geq1}}
    \sum_{m>0} 
    (2m\pi i\omega_k^{-1})^{s-1}
    \Bigl(
    \prod_{j\in J}
    \frac{1}{e^{-2m\pi i\omega_j/\omega_k}-1}
    \Bigr).
  \end{aligned}
\end{multline}
Since all $j\in I_k^+$, 
we see that the rightmost side of the above can be
continued to the whole of $\mathbb{C}$, and is equal to
\begin{multline}
  (-1)^{n-1}
  % \Bigl(
  % \prod_{\substack{j=1\\j\neq k}}^n
  % \omega_j
  % \Bigr)
  (2\pi i\omega_k^{-1})^{s-1}
  \zeta(1-s)
  \\
  % &\qquad
  +
  % \Bigl(
  % \prod_{\substack{j=1\\j\neq k}}^n
  % \omega_j
  % \Bigr)
  \sum_{m>0}
  (2m\pi i\omega_k^{-1})^{s-1}
  \biggl(
  \Bigl(
  \prod_{\substack{j=1\\j\neq k}}^n
  \frac{1}{e^{2m\pi i\omega_j/\omega_k}-1}
  \Bigr)-(-1)^{n-1}\biggl).
\end{multline}
  
For the series with $m<0$, by exchanging the roles of $I_k^+$ and $I_k^-$,
we have the same type of conclusions as follows:
If $I_k^+\neq\emptyset$,
the series corresponding to $m<0$ converges 
absolutely uniformly
for whole $s\in\mathbb{C}$,
and if $I_k^+=\emptyset$,
\begin{multline}
  \sum_{m<0}
  \Bigl( 
  \prod_{\substack{j=1\\j\neq k}}^n 
  \frac{1}
  {e^{2m\pi i\omega_j/\omega_k}-1}
  \Bigr) (2m\pi i\omega_k^{-1})^{s-1}
%   \sum_{m<0}
%   \Bigl( 
%   \prod_{\substack{j=1\\j\neq k}}^n F(2m\pi i\omega_j/\omega_k,0) 
%   \Bigr) (2m\pi i\omega_k^{-1})^{s-n}
  \\
%   \begin{aligned}
%     &=
%     \sum_{m<0} 
%     \Bigl(
%     \prod_{\substack{j=1\\j\neq k}}^n
%     (2m\pi i\omega_j/\omega_k)
%     \frac{1}{e^{2m\pi i\omega_j/\omega_k}-1}
%     \Bigr)
%     (2m\pi i\omega_k^{-1})^{s-n}
%     \\
%     &
  =
    (-1)^{n-1}
%     \Bigl(
%     \prod_{\substack{j=1\\j\neq k}}^n
%     \omega_j
%     \Bigr)
    (-2\pi i\omega_k^{-1})^{s-1}
    \zeta(1-s)
    \\
%    &\qquad
    +
%     \Bigl(
%     \prod_{\substack{j=1\\j\neq k}}^n
%     \omega_j
%     \Bigr)
    \sum_{m>0}
    (-2m\pi i\omega_k^{-1})^{s-1}
    \biggl(
    \Bigl(
    \prod_{\substack{j=1\\j\neq k}}^n
    \frac{1}{e^{-2m\pi i\omega_j/\omega_k}-1}
    \Bigr)-(-1)^{n-1}\biggl).
%  \end{aligned}
\end{multline}
\end{proof}

%%%%%%%%%%%
\section{A general formulation}\label{sec-2}
%%%%%%%%%%%

In the previous sections, we established some relations between
Barnes zeta-functions and certain series
involving hyperbolic functions (see \eqref{eq:funceq_Barnes_omake}).  
In order to study this relationship further, it is convenient to introduce a
general framework to evaluate more general series.

Let $g(t)$ be a meromorphic function on $\mathbb{C}$ which has
possible poles only on $2\pi i\mathbb{Z}$. 
For example,
we will consider $g(t)=(t/2)/\sinh(t/2)$ (see \eqref{3-6}). 

% Let $C=\{z\in \mathbb{C}\,|\,|z|=\varepsilon\}$ be a
% sufficiently small circle with radius $\varepsilon>0$ around the
% origin going counterclockwise.
% Let $C(R)=\{z\in \mathbb{C}\,|\,|z|=R\}$ be the circle
%  with radius $R>0$ around the
% origin going counterclockwise.
% Let $\widetilde{C}_0$ be a path which starts at
% $+i\infty$, passes through the left of the imaginary axis, goes around
% the origin counterclockwise and comes back to $+i\infty$ passing
% through the right of the imaginary axis.  Let $n\in \mathbb{N}$ with
% $n\geq 2$ and $\eta= e^{\pi i/n}$, that is, the primitive $2n$-th root
% of unity.  Let $C_0 = \widetilde{C}_0 \vee (-\widetilde{C}_0)$ and
% $C_l = \eta^l C_0$ $(0\leq l \leq n-1)$.  Then in
% $\mathbb{C}\setminus\bigcup_{l=0}^{n-1}\bigcup_{m\in\mathbb{Z}}D(2m\pi i\eta^l,\varepsilon)$,
% we can deform $\bigvee_{l=0}^{n-1}C_l$ into a circle with radius
% $\lambda$ satisfying $\varepsilon<\lambda<2\pi$ around the origin
% going clockwise. Therefore $C \vee \bigvee_{l=0}^{n-1} \,C_l$ is
% homologous to $0$ in
% $\mathbb{C}\setminus\bigcup_{l=0}^{n-1}\bigcup_{m\in\mathbb{Z}}D(2m\pi i\eta^l,\varepsilon)$.
Let $n\in \mathbb{N}$ with
$n\geq 2$ and $\eta= e^{\pi i/n}$, that is, the primitive $2n$-th root
of unity.
Let
\begin{equation*}
G(t)=\prod_{j=1}^{n}g(\eta^jt). 
\end{equation*}
We assume that there exist
real numbers $\gamma_1$, $\gamma_2$ with $\gamma_1>0$ and a small positive number
$\varepsilon$ such that
\begin{equation*}
  |G(t)|\leq\gamma_1 |t|^{-\gamma_2}
\end{equation*}
for all $t\in\mathbb{C}\setminus\bigcup_{l=0}^{n-1}\bigcup_{m\in\mathbb{Z}}D(2m\pi i\eta^l,\varepsilon)$.
Then we have the following theorem, which is a simple consequence of residue
calculus, but is a key result in the present paper.

\begin{theorem}\label{T-2-1}
For $h \in \mathbb{Z}$ with $h+\gamma_2>1$,
\begin{equation}
\begin{split}
& %\lim_{R\to\infty}\sum_{0<|m|\leq R}
\sum_{m\in \mathbb{Z}\setminus \{0\}}
(2m\pi i)^{-h}\left(
\sum_{l=0}^{n-1}\eta^{l(1-h)}\prod_{\substack{j=1 \\ j+l \not=n}}^{n}
g(2m\pi i\eta^{j+l})\right)
\Res_{t=2m\pi i}g(-t) \\
& \quad = -\Res_{t=0}
\left\{t^{-h}\prod_{j=1}^{n}g(\eta^j t)\right\}.
\end{split}
\label{2-1}
\end{equation}
In particular when $g$ is an even function, 
\begin{equation}
\begin{split}
& \mathcal{Z}_h
\sum_{m\in \mathbb{Z}\setminus \{0\}}
%\lim_{R\to\infty}\sum_{0<|m|\leq R}
(2m\pi i)^{-h}\left(
\prod_{j=1}^{n-1}
g(2m\pi i\eta^{j})\right)
\Res_{t=2m\pi i}g(t) \\
&\quad = -\Res_{t=0}
\left\{t^{-h}\prod_{j=0}^{n-1}g(\eta^j t)\right\},
\end{split}
\label{2-2}
\end{equation}
where
\begin{equation}
\mathcal{Z}_h= \sum_{j=0}^{n-1}\eta^{j(1-h)}=
\begin{cases}
n & \text{if $h\equiv 1$ \ {\rm (mod $2n$)}}, \\
0 & \text{if $h\not\equiv 1$ \ {\rm (mod $2n$)} \ and $2\not|\,h$},\\
\frac{2}{1-\eta^{1-h}} & \textrm{if $h\not\equiv 1$ \ {\rm (mod $2n$)} \ and $2\,|\, h$}.
\end{cases}
\label{2-3}
\end{equation}
\end{theorem}

\begin{proof}
Let $h \in \mathbb{Z}$ with $h+\gamma_2>1$.
%and $C \vee \bigvee_{l=0}^{n-1} \,C_l$ is homologous to $0$, 
 For $R\in\mathbb{N}$, 
we have
\begin{equation}
\begin{split}
\Res_{t=0}\left\{ G(t)t^{-h}\right\}&= \frac{1}{2\pi i}\int_{|t|=2\varepsilon}G(t)t^{-h}dt \\
% \Res_{t=0}\left\{ G(t)t^{-h}\right\}&= \frac{1}{2\pi i}\int_{C(2\varepsilon)}G(t)t^{-h}dt \\
%  & =\frac{1}{2\pi i}\sum_{l=0}^{n-1}\int_{\eta^l C_0}G(t)t^{-h}dt \\
%  & =\frac{1}{2\pi i}\sum_{l=0}^{n-1}\int_{C_0}G(\eta^l t)(\eta^l t)^{-h}\eta^{l}dt \\
 & =
-\sum_{l=0}^{n-1}\sum_{0<|m|\leq R} \Res_{t=2m\pi i\eta^l}\left\{G(t)t^{-h}\right\}
+\frac{1}{2\pi i}\int_{|t|=2\pi R+2\varepsilon}G(t)t^{-h}dt 
\\
 & =
-\sum_{l=0}^{n-1}\sum_{0<|m|\leq R} \Res_{t=2m\pi i}\left\{G(\eta^l t)(\eta^lt)^{-h}\eta^l\right\}
+\frac{1}{2\pi i}\int_{|t|=2\pi R+2\varepsilon}G(t)t^{-h}dt.
\end{split}
\end{equation}
%where we took all residues for $2\varepsilon<|t|<2\pi R+2\varepsilon$.
%by enlarging the circle
% by taking residues 
Since $G(t)t^{-h}\leq\gamma_1|t|^{-h-\gamma_2}$,
we have
\begin{equation}
  \begin{split}
  \Bigl\lvert\int_{|t|=2\pi R+2\varepsilon}G(t)t^{-h}dt\Bigr\rvert
  &\leq
  \int_{|t|=2\pi R+2\varepsilon}|G(t)t^{-h}||dt|
  \\
  &\leq
  \gamma_1\int_{|t|=2\pi R+2\varepsilon}|t|^{-h-\gamma_2}|dt|
  \\
  &\leq 2\pi\gamma_1(2\pi R+2\varepsilon)^{-h-\gamma_2+1}
  \\
  &\to 0
\end{split}
\end{equation}
as $R\to\infty$.
Hence by letting $R\to\infty$, we obtain
\begin{equation}
  \Res_{t=0}\left\{ G(t)t^{-h}\right\}
% & =-\lim_{R\to\infty}\sum_{l=0}^{n-1}\eta^{l(1-h)}\sum_{0<|m|\leq R} \Res_{t=2m\pi i}\left\{t^{-h}G(\eta^l t)\right\},
   =-\sum_{l=0}^{n-1}\eta^{l(1-h)}\sum_{m\in \mathbb{Z}\setminus \{ 0\}} \Res_{t=2m\pi i}\left\{G(\eta^l t)t^{-h}\right\}
%\label{2-4}
\end{equation}
%where we used the fact that
because
%by $G(t)t^{-h}\leq\gamma_1|t|^{-h-\gamma_2}$,
%we have
% \begin{equation}
%   \begin{split}
%   \Bigl\lvert\int_{|t|=2\pi R+2\varepsilon}G(t)t^{-h}dt\Bigr\rvert
%   &\leq
%   \int_{|t|=2\pi R+2\varepsilon}|G(t)t^{-h}||dt|
%   \\
%   &\leq
%   \gamma_1\int_{|t|=2\pi R+2\varepsilon}|t|^{-h-\gamma_2}|dt|
%   \\
%   &\leq 2\pi\gamma_1(2\pi R+2\varepsilon)^{-h-\gamma_2+1}
%   \\
%   &\overset{R\to\infty}{\to}0
% \end{split}
% \end{equation}
% and
\begin{equation}
  \begin{split}
    \Bigl\lvert\Res_{t=2m\pi i}G(\eta^lt)t^{-h}\Bigr\rvert
    &\leq\frac{1}{2\pi}\int_{|t-2m\pi i|=2\varepsilon}|G(\eta^lt)t^{-h}||dt|
    \\
    &\leq\frac{\gamma_1}{2\pi}\int_{|t-2m\pi i|=2\varepsilon}|t|^{-h-\gamma_2}|dt|
    \\
    &\leq 
    2\varepsilon\gamma_1(2\pi |m|-2\varepsilon)^{-h-\gamma_2}
  \end{split}
\end{equation}
and hence the convergence is absolute.

Since $\eta$ is the primitive $2n$-th root of unity, we see that for
$l\in \mathbb{Z}$ with $0\leq l\leq n-1$, the residue of
$$t^{-h}G(\eta^l t)=t^{-h}\prod_{j=1}^{n}g( \eta^{j+l}t)$$
at $t=2m\pi i$ is equal to
$$(2m\pi i)^{-h} \prod_{\substack{j=1 \\ j+l\not=n}}^{n}g( 2m\pi i\eta^{j+l})\times \Res_{t=2m\pi i}g(-t),$$
which gives \eqref{2-1}. 

In particular when $g$ is even, we have $G(\eta^lt)=
G(t)$ $(0\leq l\leq n-1)$, because $g(\eta^{r}t)=g(-\eta^{r-n}t)=g(\eta^{r-n}t)$ for $n+1\leq r<2n$. Therefore
$$\sum_{l=0}^{n-1}\eta^{l(1-h)}\sum_{m\in \mathbb{Z}\setminus \{ 0\}} 
\Res_{t=2m\pi i}\left\{t^{-h}G(\eta^l t)\right\}=\mathcal{Z}_h 
\sum_{m\in \mathbb{Z}\setminus \{ 0\}} \Res_{t=2m\pi i}\left\{t^{-h}G(t)\right\}.$$
This completes the proof.
\end{proof}

The same idea as in the above proof is also used in \cite{KN} in a
different situation.  It is to be noted that the original
method of Cauchy \cite{Ca} is essentially similar.

%%%%%%%%%%%%%%%%%%%%%%%%%%%%%%
\section{Explicit formulas} \label{sec-3}
%%%%%%%%%%%%%%%%%%%%%%%%%%%%%%

We recall the Bernoulli polynomials $\{B_j(y)\}$ defined by \eqref{Ber-poly}:
\begin{equation*}
F(t,y)=\frac{te^{ty}}{e^t-1}=\sum_{j=0}^\infty B_j(y)\frac{t^j}{j!}.
\end{equation*}
For $y \in \mathbb{R}$, let 
\begin{equation}
\begin{split}
H(t,y)& =\frac{F(t,y)+F(-t,y)}{2} =\frac{t}{2}\frac{e^{t(y-1/2)}+e^{-t(y-1/2)}}{e^{t/2}-e^{-t/2}}\\
      & =\frac{t}{2}\frac{\cosh(t(y-1/2))}{\sinh(t/2)}.
\end{split}
\label{3-2}
\end{equation}
Then we see that 
\begin{equation}
H(t,y)=\sum_{m=0}^\infty B_{2m}\left(y \right)\frac{t^{2m}}{(2m)!}. \label{3-3}
\end{equation}
It follows from \eqref{3-2} that $H(t,y)$ has simple poles at $t=2m\pi
i$ $(m\in \mathbb{Z}\setminus \{0\})$ and its residue is
\begin{equation}
\Res_{t=2m\pi i} H(t,y)=\frac{2m\pi i}{2} \frac{e^{2m\pi i(y-1/2)}+e^{-2m\pi i(y-1/2)}}{(-1)^m} =2m\pi i\cos(2m\pi y). \label{3-4}
\end{equation}
By \eqref{2-2}, we have the following result which includes the known
formulas \eqref{eq-1-2} and \eqref{eq-1-3} given by Cauchy, Mellin,
Ramanujan and so on (see Section \ref{sec-1}).

\begin{theorem} \label{P-3-1}
Assume $0<y<1$ and $p \in \mathbb{Z}$,
or $y=0,1$ and $p>n/2$.
For $n \in \mathbb{N}$ with $n\geq 2$ and $\eta=e^{\pi i/n}$, 
\begin{multline}
  \mathcal{Z}_{2p+1}\sum_{m\in \mathbb{Z}\setminus \{0\}}
\frac{\cos(2m\pi y)}{m^{2p+1-n}} \left(
\prod_{j=1}^{n-1}
\frac{\cosh\left( 2m\pi i\eta^{j}(y-1/2)\right)}{\sinh\left( m\pi i\eta^{j}\right)}\right)
 \\
 = -\frac{2^{n-1}(2\pi i)^{2p+1-n}}{\eta^{n(n-1)/2}}
\sum_{\substack{m_1,\ldots,m_n \geq 0 \\ m_1+\cdots+m_n=p}} \prod_{\nu=1}^{n} \frac{B_{2m_\nu}(y)}{(2m_\nu) !}\eta^{2(\nu-1)m_\nu},
\label{3-5}
\end{multline}
where $\mathcal{Z}_h$ is defined by \eqref{2-3}. 
Furthermore, 
assume $0<y<1$ and $p\geq n/2$. %and $p>n/2$ if $y=0,1$.
Then the both sides of \eqref{3-5} is also equal to
\begin{equation}
\label{eq:sum_zetas}
%\frac{(2\pi i)^{2p+1-n}(-1)^{2p+1-n}}{2(2p-n)!}
\frac{(2\pi i)^{2p+1-n}(-1)^{1-n}}{2(2p-n)!}
  \sum_{y_0\in\{y,1-y\}}\cdots
  \sum_{y_{n-1}\in\{y,1-y\}}
  \zeta_n(n-2p,\sum_{j=0}^{n-1}\eta^jy_j;1,\eta,\ldots,\eta^{n-1}).
\end{equation}
\end{theorem}

\begin{proof}
  Since $H(-t,y)=H(t,y)$, we can apply \eqref{2-2} with $g(t)=H(t,y)$
  and $h=2p+1$.  In this case $\gamma_2$ is arbitrarily large for
  $0<y<1$ and $\gamma_2=-n$ for $y=0,1$, and hence the condition
  $h+\gamma_2>1$ is satisfied because we assume $p>n/2$ if $y=0,1$.
  By \eqref{3-3}, we have
  \begin{equation*}
    H(\eta^j t,y)=\sum_{m=0}^\infty B_{2m}(y)\frac{\eta^{2jm}t^{2m}}{(2m)!}.
  \end{equation*}
  Hence we obtain 
  \begin{equation*}
    \Res_{t=0}\left\{t^{-2p-1}\prod_{j=0}^{n-1}H(\eta^j t,y)\right\} 
    =\sum_{\substack{m_1,\ldots,m_n\geq 0 \\ m_1+\cdots+m_n=p}} \prod_{\nu=1}^{n}\frac{B_{2m_\nu}(y)}{(2m_\nu)!}\eta^{2(\nu-1)m_\nu}.
  \end{equation*}
  Therefore, by using \eqref{3-2} and \eqref{3-4}, we have
\begin{equation*}
\begin{split}
& \mathcal{Z}_{2p+1} \sum_{m\in \mathbb{Z}\setminus \{0\}}
(2m\pi i)^{-2p-1}\left(
\prod_{j=1}^{n-1}
m\pi i\eta^{j}\frac{\cosh\left( 2m\pi i\eta^{j}(y-1/2)\right)}{\sinh\left( m\pi i\eta^{j}\right)}\right)
2m\pi i\cos(2m\pi y) \\
& \quad = -\sum_{\substack{m_1,\ldots,m_n \geq 0 \\ m_1+\cdots+m_n=p}} \prod_{\nu=1}^{n} \frac{B_{2m_\nu}(y)}{(2m_\nu) !}\eta^{2(\nu-1)m_\nu}.
\end{split}
\end{equation*}
Thus we obtain \eqref{3-5}. 

Assume $0<y<1$ and $p\geq n/2$. % and $p>n/2$ if $y=0,1$.
Then $h=2p+1\geq n+1$. % if $0<y<1$ and $h>n+1$ if $y=0,1$.
% We see that \eqref{3-5} can be
% regarded as a sum of special values of Barnes zeta-functions.
% In the following, we show it.
%First 
Note that
\begin{equation*}
  F(-t,y)=\frac{-te^{-ty}}{e^{-t}-1}=\frac{te^{t(1-y)}}{e^t-1}
=F(t,1-y).
\end{equation*}
Then we have
\begin{equation}
  \begin{split}
  G(t)&=\prod_{j=0}^{n-1}H(\eta^j t,y)=
  \prod_{j=0}^{n-1}\frac{F(\eta^j t,y)+F(\eta^j t,1-y)}{2}
  \\
  &=2^{-n}\sum_{y_0\in\{y,1-y\}}\cdots
  \sum_{y_{n-1}\in\{y,1-y\}}
  \prod_{j=0}^{n-1}F(\eta^j t,y_j).
\end{split}
\end{equation}
Hence
\begin{multline}
\label{namaewotuketa}
\begin{aligned}
  &\Res_{t=0}\left\{G(t)t^{-h}\right\}\\
%  \begin{aligned}
  &\quad=\frac{2^{-n}}{2\pi i}
  \sum_{y_0\in\{y,1-y\}}\cdots
  \sum_{y_{n-1}\in\{y,1-y\}}
  \int_{|t|=\varepsilon}
  \Bigl(\prod_{j=0}^{n-1}F(\eta^j t,y_j)\Bigr)
  t^{(n+1-h)-n-1}dt.
\end{aligned}
\end{multline}
Since $0<y<1$, we see that 
$1,\eta,\ldots,\eta^{n-1}$ and $\sum_{j=0}^{n-1}\eta^j(1-y_j)$ are belonging to
the half plane $H(\theta_n)$, where $\theta_n=\pi/2-\pi/(2n)$.   Therefore,
deforming the path $|t|=\varepsilon$ to $C(\theta_n)$, we find that each integral
on the right-hand side of \eqref{namaewotuketa} is of the same form as the
integral on the right-hand side of \eqref{eq:int_rep_Barnes} with 
$s=n+1-h$.   Using \eqref{eq:int_rep_Barnes} and \eqref{gamma-residue},
we obtain that the right-hand side of \eqref{namaewotuketa} is equal to
\begin{equation*}
  \begin{split}
    &\frac{2^{-n}\eta^{n(n-1)/2}}{2\pi i}
    \sum_{y_0\in\{y,1-y\}}\cdots
    \sum_{y_{n-1}\in\{y,1-y\}}
    \\
    &\qquad
    % \lim_{s\to n+1-h}
    % \Gamma(s)(e^{2\pi is}-1)
    \frac{2\pi i(-1)^{h-n-1}}{(h-n-1)!}
    \zeta_n(n+1-h,\sum_{j=0}^{n-1}\eta^j(1-y_j);1,\eta,\ldots,\eta^{n-1})
    \\
    &=
    \frac{2^{-n}(-1)^{h-n-1}\eta^{n(n-1)/2}}{(h-n-1)!}
    \\
    &\qquad\times
    \sum_{y_0\in\{y,1-y\}}\cdots
    \sum_{y_{n-1}\in\{y,1-y\}}
    \zeta_n(n+1-h,\sum_{j=0}^{n-1}\eta^jy_j;1,\eta,\ldots,\eta^{n-1}),
  \end{split}
\end{equation*}
which implies \eqref{eq:sum_zetas}.
\end{proof}

%%%%%%%%%%%%%%%%%%%%
In particular when $y=\frac{1}{2}$ in \eqref{3-5} and
\eqref{eq:sum_zetas}, we have the following formula, which can be
regarded as a multiple generalization of \eqref{eq-1-2}.

\begin{corollary} \label{C-3-1}
For $p \in \mathbb{Z}$, $n \in \mathbb{N}$ with $n\geq 2$ and $\eta=e^{\pi i/n}$, 
\begin{equation}
\begin{split}
&  \mathcal{Z}_{2p+1}\sum_{m\in \mathbb{Z}\setminus \{0\}}
\frac{(-1)^m}{\left( \prod_{j=1}^{n-1}{\sinh\left( m\pi i\eta^{j}\right)}\right)m^{2p+1-n}} \\
& \quad = -\frac{2^{n-1}(2\pi i)^{2p+1-n}}{\eta^{n(n-1)/2}}
\sum_{\substack{m_1,\ldots,m_n \geq 0 \\ m_1+\cdots+m_n=p}} \prod_{\nu=1}^{n} \frac{B_{2m_\nu}(1/2)}{(2m_\nu) !}\eta^{2(\nu-1)m_\nu}.
\end{split}
\label{3-6}
\end{equation}
Furthermore if $p\geq n/2$, then \eqref{3-6} is equal to
\begin{equation}
  \frac{(2\pi i)^{2p+1-n}(-1)^{1-n}2^{n-1}}{2(2p-n)!}
  \zeta_n(n-2p,1/(1-\eta);1,\eta,\ldots,\eta^{n-1}).
\end{equation}
\end{corollary}

Now we consider the case $n=2$, $\eta=e^{\pi i/2}=i$ and $p=2k+2$ in \eqref{3-5}. Then we obtain the following.

\begin{corollary} \label{C-3-2}
For $k \in \mathbb{N}_0$ and $0\leq y\leq 1$,
\begin{equation}
\begin{split}
&  \sum_{m\in \mathbb{Z}\setminus \{0\}}
\frac{\cos(2m\pi y)}{m^{4k+3}} \left(
\frac{\cosh\left( 2m\pi (y-1/2)\right)}{\sinh\left( m\pi \right)}\right)
 \\
& \quad = {(2\pi)^{4k+3}}\sum_{j=0}^{2k+2}(-1)^{j+1} \frac{B_{2j}(y)}{(2j)!}\frac{B_{4k+4-2j}(y)}{(4k+4-2j)!}.
\end{split}
\label{3-7}
\end{equation}
In particular when $y=0$ and $y=\frac{1}{2}$, we obtain \eqref{eq-1-3} and \eqref{eq-1-2}, respectively.
\end{corollary}

Next we consider the case $n=3$. Let $\eta=e^{\pi i/3}=-\rho^2$ with $\rho=e^{2\pi i/3}$ and $p=3k+3$. Then we have $\eta^{3(3-1)/2}=-1$ and $\mathcal{Z}_{6(k+1)+1}=3$ by \eqref{2-3}. From \eqref{3-5}, we have 
\begin{equation}
\begin{split}
&  \sum_{m\in \mathbb{Z}\setminus \{0\}}
\frac{\cos(2m\pi y)}{m^{6k+4}} \left(
\frac{\cosh\left( 2m\pi i(-\rho^2)(y-1/2)\right)\cosh\left( 2m\pi i\rho(y-1/2)\right)}{\sinh\left( m\pi i(-\rho^2)\right)\sinh\left( m\pi i\rho\right)}\right)
 \\
& \quad = -\frac{2^{2}(2\pi i)^{6k+4}}{-3}
\sum_{\substack{m_1,m_2,m_3 \geq 0 \\ m_1+m_2+m_3=3k+3}} \prod_{\nu=1}^{3} \frac{B_{2m_\nu}(y)}{(2m_\nu) !}\rho^{2(\nu-1)m_\nu}.
\end{split}
\label{3-8}
\end{equation}
Note that $\rho^2=1/\rho$ and $\rho=-1/\rho-1$. Hence, by using 
\begin{align*}
& \sinh(m\pi i \rho)=-(-1)^m\sinh(m\pi i/\rho),
\end{align*}
we can rewrite \eqref{3-8} as follows.

\begin{corollary} \label{C-3-3}
  For $k \in \mathbb{N}_0$ and $0\leq y\leq 1$,
\begin{equation}
\begin{split}
&  \sum_{m\in \mathbb{Z}\setminus \{0\}}
\frac{(-1)^m\cos(2m\pi y)}{m^{6k+4}} \left(
\frac{\cosh\left( (2m\pi i/\rho)(y-1/2)\right)\cosh\left( 2m\pi i\rho(y-1/2)\right)}{\sinh\left( m\pi i/\rho\right)^2}\right)
 \\
& \quad = \frac{4(-1)^{k}(2\pi)^{6k+4}}{3}
\sum_{\substack{m_1,m_2,m_3 \geq 0 \\ m_1+m_2+m_3=3k+3}} \frac{B_{2m_1}(y)}{(2m_1) !}\frac{B_{2m_2}(y)}{(2m_2) !}\frac{B_{2m_3}(y)}{(2m_3) !}\rho^{m_2+2m_3}.
\end{split}
\label{3-9}
\end{equation}
In particular when $y=0$ and $y=\frac{1}{2}$, the following equations hold:
\begin{align}
& \sum_{m\in \mathbb{Z}\setminus \{0\}}\frac{\coth(m\pi i/\rho)^2}{m^{6k+4}} \label{3-10}\\
& \quad = \frac{4(-1)^{k}(2\pi)^{6k+4}}{3}
\sum_{\substack{m_1,m_2,m_3 \geq 0 \\ m_1+m_2+m_3=3k+3}} \frac{B_{2m_1}(0)}{(2m_1) !}\frac{B_{2m_2}(0)}{(2m_2) !}\frac{B_{2m_3}(0)}{(2m_3) !}\rho^{m_2+2m_3}, \notag\\
& \sum_{m \in \mathbb{Z}\setminus \{0\}} \frac{1}{\sinh(m\pi i/\rho)^2m^{6k+4}}  \label{3-11}\\
& \quad = \frac{4(-1)^{k}(2\pi)^{6k+4}}{3}
\sum_{\substack{m_1,m_2,m_3 \geq 0 \\ m_1+m_2+m_3=3k+3}} \frac{B_{2m_1}(1/2)}{(2m_1) !}\frac{B_{2m_2}(1/2)}{(2m_2) !}\frac{B_{2m_3}(1/2)}{(2m_3) !}\rho^{m_2+2m_3}. \notag
\end{align}
\end{corollary}

\begin{example} \label{Ex-3-4}
From \eqref{3-10} and \eqref{3-11}, we obtain
\begin{align} 
& \sum_{m\not=0}\frac{\coth(m\pi i/\rho)^2 }{m^4} =\frac{62}{2835}\pi^4, \label{3-12}\\
& \sum_{m\not=0}\frac{\coth(m\pi i/\rho)^2 }{m^{10}} =\frac{40247}{1915538625}\pi^{10}, \label{3-13}\\
& \sum_{m\not=0}\frac{1}{\sinh(m\pi i/\rho)^2 m^4} =-\frac{1}{2835}\pi^4, \label{3-14}\\
& \sum_{m\not=0}\frac{1}{\sinh(m\pi i/\rho)^2 m^{10}} =-\frac{703}{1915538625}\pi^{10}. \label{3-15}
\end{align}
Additionally, by setting $(p,n)=(4,4)$, $(5,5)$ in equation \eqref{3-6}, we obtain
\begin{align}
& \sum_{m\not=0}\frac{(-1)^m}{\sinh(m\pi)\sinh(m\pi i\zeta_8)\sinh(m\pi i\zeta_8^{-1})m^5} =\frac{1}{37800}\pi^5, \label{3-15-2}\\
& \sum_{m\not=0}\frac{(-1)^m}{\sinh(m\pi i\zeta_5)\sinh(m\pi i\zeta_5^2)\sinh(m\pi i\zeta_5^{3})\sinh(m\pi i\zeta_5^4)m^6} =-\frac{1}{467775}\pi^6, \label{3-15-3}
\end{align}
where $\zeta_k=e^{2\pi i/k}$ $(k\in \mathbb{N})$.
\end{example}

\begin{remark} \label{R-1}
By using the same method as introduced in this paper, 
we can recover the known formulas, for example, 
\begin{align*}
& \sum_{m=0}^\infty \frac{(-1)^{m}}{\cosh ((2m+1)\pi/2)((2m+1)/2)^{4k+1}} \\
& \quad\quad =\frac{(2\pi)^{4k+1}}{8}\sum_{j=0}^{2k}(-1)^j \frac{E_{2j}(1/2)}{(2j)!}\frac{E_{4k-2j}(1/2)}{(4k-2j)!},\\
& \sum_{m=0}^\infty \frac{(-1)^{m}}{\cosh ((2m+1)\sqrt{3}\pi/2)((2m+1)/2)^{6k+1}}\notag \\
& \quad\quad =\frac{(-1)^{k+1}(2\pi)^{6k+1}}{2}\sum_{j=0}^{3k} \frac{E_{2j+1}(0)}{(2j+1)!}\frac{B_{6k-2j}(0)}{(6k-2j)!}\cos\left(\frac{(2j+1)\pi}{3}\right)  
\end{align*}
for $k\in \mathbb{N}_{0}$ (see Watson \cite{Wa} and Berndt \cite{Be1}), where $\{ E_n(x)\}$ are the Euler polynomials defined by
\begin{align*}
& \frac{2e^{xt}}{e^t+1}=\sum_{n=0}^\infty E_n(x)\frac{t^n}{n!}.
\end{align*}
More generally, we can give relevant analogues of these results like those in Example \ref{Ex-3-4}.
\end{remark}

\ 

\noindent
{Y. Komori\\
Department of Mathematics, Rikkyo University, 3-34-1 Nishi-Ikebukuro, Toshima-ku, Tokyo 171-8501, Japan \\ email:\ {\rm komori@rikkyo.ac.jp} }

\smallskip

\noindent
{K. Matsumoto\\
Graduate School of Mathematics, Nagoya University, Chikusa-ku, Nagoya 464-8602 Japan \\ email:\ {\rm kohjimat@math.nagoya-u.ac.jp}}

\smallskip

\noindent
{H. Tsumura\\
Department of Mathematics and Information Sciences, Tokyo Metropolitan University,\\ 1-1, Minami-Ohsawa, Hachioji, Tokyo 192-0397, Japan \\ email:\ {\rm tsumura@tmu.ac.jp}}


\begin{thebibliography}{3}
%
% and use \bibitem to create references. Consult the Instructions
% for authors for reference list style.
%
% Format for Journal Reference

\bibitem{Ar}
\sc{T. Arakawa},
{\it Dirichlet series $\sum_{n=1}^{\infty}(\cot\pi n\alpha)/n^s$, Dedekind sums, and Hecke $L$-functions for real quadratic fields},
Comment. Math. Univ. St. Pauli \textbf{37} (1988), 209-235.

\bibitem{B01}
\sc{E. W. Barnes}, 
{\it The theory of the double gamma function},
Phil. Trans. Roy. Soc. (A) \textbf{196} (1901), 265-387.

\bibitem{B04}
\sc{E. W. Barnes},
{\it On the theory of multiple gamma function},
Trans. Cambridge Phil. Soc. \textbf{19} (1904), 374-425.

\bibitem{BeTAMS}
\sc{B. C. Berndt},
{\it Generalized Dedekind eta-functions and generalized Dedekind sums},
Trans. Amer. Math. Soc. \textbf{178} (1973), 495-508.

\bibitem{Be0}
\sc{B. C. Berndt},
{\it Generalized Eisenstein series and modified Dedekind sums},
J. Reine Angew. Math., \textbf{272} (1974), 182--193.

\bibitem{BeRMJM}
\sc{B. C. Berndt},
{\it Modular transformations and generalizations of several formulae of Ramanujan},
Rocky Mountain J. Math. \textbf{7} (1977), 147-189.

\bibitem{Be1}
\sc{B. C. Berndt},
{\it Analytic Eisenstein series, theta-functions, and series relations in the spirit of Ramanujan},
J. Reine Angew. Math., \textbf{303/304} (1978), 332--365.

\bibitem{Be2}
\sc{B. C. Berndt},
{``Ramanujan's Notebooks, part II''},
Springer-Verlag, New-York, 1989.

\bibitem{Be3}
\sc{B. C. Berndt},
{``Ramanujan's Notebooks, part V''},
Springer-Verlag, New-York, 1998.

\bibitem{Ca}
\sc{A. L. Cauchy},
{``Exercices de Math{\'e}matiques''}, Paris, 1827;
{``Oeuvres Completes D'Augustin Cauchy''},
S\'erie II, t. VII, Gauthier-Villars, Paris, 1889.

\bibitem{Di}
\sc{K. Dilcher},
{\it Zeros of Bernoulli, generalized Bernoulli and Euler polynomials},
Memoirs of Amer. Math. Soc., \textbf{386} (1988).

\bibitem{Eg}
\sc{S. Egami},
{``An elementary theory of multiple zeta functions''}, 
Lecture Note (in Japanese), Niigata University, 1998.

\bibitem{Ha}
\sc{G. H. Hardy and J. E. Littlewood}, 
{\it Some problems of Diophantine approximation: The lattice-points of a right-angled triangle}, 
Proc. London Math. Soc., (2) \textbf{20} (1922), 15--36.

\bibitem{KN}
\sc{Y. Komori and M. Noumi}, 
in preparation.

\bibitem{Ma98}
\sc{K. Matsumoto},
{\it Asymptotic series for double zeta, double gamma, and Hecke $L$-functions},
Math. Proc. Cambridge Phil. Soc. \textbf{123} (1998), 385-405;
\textit{Corrigendum and addendum}, ibid. \textbf{132} (2002), 377-384.

\bibitem{Me1}
\sc{Hj. Mellin}, 
{\it Eine Formel f\"ur den Logarithmus transcendenter Funktionen 
von endlichem Geschlecht}, Acta Soc. Sci. Fennicae, \textbf{29}, n.4 (1902), 49pp. 

\bibitem{Me2}
\sc{Hj. Mellin}, 
{\it Eine Formel f\"ur den Logarithmus transcendenter Funktionen 
von endlichem Geschlecht}, 
Acta Math., \textbf{25}, n.1 (1902), 165--183. 

\bibitem{Wa}
\sc{G. N. Watson},
{\it Theorems stated by Ramanujan II}, J. London Math. Soc., \textbf{3} (1928), 216--225.

\end{thebibliography}
\end{document}